\documentclass[12pt]{article}
\usepackage{style}

\usepackage[english]{babel}
\usepackage{amsthm,amsmath,amssymb}
\usepackage{graphicx}
\usepackage[small,bf,hang]{caption} % Improved caption
\usepackage{tabularx}

\usepackage{longtable}
\usepackage{subfig}                      % Figures which consist of multiple subfigures
\usepackage{tikz}

\usepackage{algorithm}                   % Om pseudocode voor te stellen in latex
\usepackage{algpseudocode}

\usepackage{hyperref}
\usepackage{enumerate} % Numbered lists

\usepackage{multirow}

\newcommand{\ram}[2]{R(J_{#1},J_{#2})}
\newcommand{\RamSet}{\mathcal{R}}
\newcommand{\N}{\mathbb{N}}
\newcommand{\range}[2][1]{\{#1,\dots,#2\}}
\newcommand{\neigh}[3]{N_{#1}^{#2}(#3)}

\newcolumntype{L}[1]{>{\raggedright\arraybackslash}p{#1}}
\newcolumntype{C}[1]{>{\centering\arraybackslash}p{#1}}
\newcolumntype{R}[1]{>{\raggedleft\arraybackslash}p{#1}}

\graphicspath{{figures/}}

\MSC{05C55, 05C30, 05C35, 05D10, 68R10}

\title{New bounds for Ramsey numbers $R(K_k-e, K_l-e)$}

\author{
Jan Goedgebeur\\
\small Department of Computer Science\\[-0.8ex]
\small KU Leuven Campus Kulak\\[-0.8ex]
\small 8500 Kortrijk, Belgium\\
\small Department of Applied Mathematics, Computer Science and Statistics\\[-0.8ex]
\small Ghent University\\[-0.8ex]
\small 9000 Ghent, Belgium\\[-0.8ex]
\small\tt jan.goedgebeur@ugent.be\\
\\
Steven Van Overberghe \\
\small Department of Applied Mathematics, Computer Science and Statistics\\[-0.8ex]
\small Ghent University\\[-0.8ex]
\small 9000 Ghent, Belgium\\[-0.8ex]
\small\tt steven.vanoverberghe@ugent.be\\
}

\begin{document}

\date{} % delete this line to display the current date
\maketitle

\begin{abstract}
Let \( R(H_1,H_2) \) denote the Ramsey number for the graphs \( H_1, H_2 \), and let \( J_k \) be \( K_k{-}e \).  
We present algorithms which enumerate all circulant and block-circulant Ramsey graphs for different types of graphs, thereby obtaining several new lower bounds on Ramsey numbers including: \( 49 \leq R(K_3,J_{12}) \), \( 36 \leq R(J_4,K_8) \), \( 43 \leq R(J_4,J_{10}) \), \( 52 \leq R(K_4,J_8) \), \( 37 \leq \ram{5}{6} \), \( 43 \leq R(J_5,K_6) \), \( 65\leq\ram{5}{7} \). 
We also use a gluing strategy to derive a new upper bound on \( \ram{5}{6} \). With both strategies combined, we prove the value of two Ramsey numbers: \( \ram{5}{6}=37 \) and \( \ram{5}{7}=65 \). We also show that the 64-vertex extremal Ramsey graph for \( \ram{5}{7} \)  is unique. 
Furthermore, our algorithms also allow to establish new lower bounds and exact values on Ramsey numbers involving wheel graphs and complete bipartite graphs, including: $R(W_7,W_4) = 21$, $R(W_7,W_7) = 19$, $R(K_{3,4},K_{3,4}) = 25$, and $R(K_{3,5}, K_{3,5})=33$.

  % keywords are optional
  \bigskip\noindent \textbf{Keywords:} Ramsey number, (block-)circulant graph,  almost-complete graph, wheel graph, computation
\end{abstract}

%--------------------------------------------------------------------------------------------------

%\newpage
\section{Introduction}
\label{section:intro}
In this paper all graphs are simple and undirected. A graph $G=(V,E)$ consists of a set of vertices $V$ and a set of edges $E$. A graph $G' = (V',E')$ is a \textit{subgraph} of $G$ if $V' \subseteq V$ and $E' \subseteq E$. If $G'$ is a subgraph of $G$ and $\forall\, v, w \in V'$ the following holds: $\{v,w\} \in E \Rightarrow \{v,w\} \in E'$, then $G'$ is called an \textit{induced} subgraph of $G$. In that case we also refer to $G'$ as the subgraph of $G$ induced by the set of vertices $X = V'$ (denoted by~$G[X]$). We refer to~\cite{graph_theory_diestel} for any standard graph theory terminology which is not explicitly defined here. 

For two graphs $H_1$, $H_2$, the \textit{Ramsey number \(R(H_1,H_2)\)} is defined as the smallest integer~$n$ such that every assignment of two colours (e.g.\ blue and red) to the edges of the complete graph $K_n$ contains $H_1$ as a blue subgraph, or $H_2$ as a red subgraph.
A two-coloured $K_n$ containing no blue copy of $H_1$ nor a red copy of $H_2$ is called an \( (H_1,H_2;n) \)-(Ramsey)-graph.
The set of all \( (H_1,H_2;n) \)-graphs is denoted by \( \RamSet(H_1,H_2;n) \).
\( (H_1,H_2;n) \)-graphs with \( n= R(H_1,H_2) -1 \) are called \textit{extremal} Ramsey graphs for $R(H_1,H_2)$.  
The concept of Ramsey numbers also generalises to $c$ colours (i.e.\  $R(H_1,H_2,...,H_c)$), but in this article we focus on two colours.

The most-studied Ramsey numbers are those where \( H_1\) and \(H_2 \) are complete graphs, these are also called the \textit{classical} Ramsey numbers. In this paper we will mainly study Ramsey numbers involving \( J_k:=K_k{-}e \), i.e.\ complete graphs with one edge removed. 

Finding the exact value of Ramsey numbers is a challenging problem. This line of research already started in 1955 with the computation of \(R(K_3,K_4)\) and \(R(K_3,K_5)\)~\cite{greenwood1955combinatorial}, but in the meantime only a handful of new classical Ramsey numbers have been fully determined. For most cases, there are only lower and upper bounds.

For classical Ramsey numbers, the last exact value was determined by McKay and Radziszowski in 1995~\cite{R45} when they showed that \( R(K_4,K_5)=25 \). The most recent improvement on a bound for a (small) classical Ramsey number dates from 2018 when Angeltveit and McKay~\cite{angeltveitmckay} improved the upper bound $R(K_5,K_5) \leq 49$ (which had been standing since 1997~\cite{mckay1997subgraph}) to 48.

For Ramsey numbers involving \( J_k \), the most recently obtained exact value was \( R(K_3,J_{10})=37 \), which was determined by Goedgebeur and Radziszowski in 2013~\cite{staszek13-2}.

An overview of more values and bounds for Ramsey numbers with small parameters can be found in Radziszowski's dynamic survey~\cite{dynSur}.

The lower bounds on small Ramsey numbers are very often derived from Ramsey graphs which result from heuristic searches, for example using simulated annealing or tabu search as was done in~\cite{exoo2012ramsey, exoo2013some}. 
These methods are suitable to find graphs that have no apparent structure, but often fail to generalise to larger cases. In this article however, we will not use heuristic algorithms. Instead we will search for Ramsey graphs in a more constructive and exhaustive way by designing efficient algorithms to generate all circulant and block-circulant Ramsey graphs for various parameters.

This article is organised as follows. In Section~\ref{sect:UB} we improve the upper bound of \(\ram{5}{6}\) from 38 (which was established by Lidick{\'y} and Pfender in~\cite{SemiDef})  to 37 using a gluing technique. Edge-counting restrictions will allow to do this without the help of computers. 

In Section~\ref{sect:LB} we present exhaustive algorithms to generate circulant and block-circular Ramsey graphs. Using our implementation of these algorithms, we manage to improve lower bounds on a large variety of Ramsey numbers. By combining this with the new upper bounds from Section~\ref{sect:UB}, we determine the value of two new Ramsey numbers: \( \ram{5}{6}=37 \) and \( \ram{5}{7}=65 \). (The previous bounds for these Ramsey numbers were $31 \leq \ram{5}{6} \leq 38$ and $40 \leq \ram{5}{7} \leq 65$ and these previous lower and upper bounds were established in~\cite{Exoo2000} and~\cite{SemiDef}, respectively).
Moreover, we also show that the 64-vertex extremal Ramsey graph for \( \ram{5}{7} \)  is unique. 

Our algorithms also allow to establish new lower bounds and exact values of Ramsey numbers involving wheel graphs and complete bipartite graphs, including: $R(W_7,W_4) = 21$, $R(W_7,W_7) = 19$, $R(K_{3,4},K_{3,4})=25$, and $R(K_{3,5}, K_{3,5})=33$, which is also discussed in more detail in Section~\ref{sect:LB}.

Finally, we end this article in Section~\ref{sect:further_research} with an open problem and some suggestions for further research.

\section{Improving upper bounds on Ramsey numbers}
\label{sect:UB}

The primary goal of this section is to improve the upper bound on \( \ram{5}{6} \) from 38 to 37, for which we will use a variation of the well-known gluing method. The previous upper bound was established by Lidick{\'y} and Pfender in~\cite{SemiDef}.

\subsection{Definitions and preliminaries}
\label{subsect:UB_prelim}

Let \( G=(V,E) \) be a complete graph whose edges are 2-coloured by the function $c:E\rightarrow\{1,2\}$. We now define:
\begin{itemize}
	\item  \( \neigh{G}{i}{v} := G[\{w \in V \ |\ c(\{v,w\})=i \}]\), i.e.\ the subgraph of $G$ induced by all vertices $w$ adjacent to $v$ for which the edge $\{v,w\}$ has colour  $i$, with the same colouring $c$.
	\item \( \deg_i(v) := |\neigh{G}{i}{v}| \)
	\item \( e_{i}(G) := |\{e \in E\ | \ c(e)=i \}| \)
	\item \(\bar n_j :=|\{v\in V \ |\ \deg_1(v)=j \}| \), i.e.\ the number of vertices having degree $j$ in colour~1.
	\item \( E_{i}(J_k,J_l;n) := \max\{ e_{i}(G)\  |\  G \in \RamSet(J_k,J_l;n) \} \), for \( n < R(J_k,J_l) \)
\end{itemize}

Let $G$ be a \( (J_k,J_l;n) \)-graph, then for an arbitrary vertex $v$, \( \neigh{G}{1}{v} \) is a \( (J_{k-1},J_l;n_1) \)-graph and \( \neigh{G}{2}{v} \) is a \( (J_{k},J_{l-1};n_2) \)-graph, where \( n_1+n_2=n-1 \).
Gluing methods reverse this process and go from these neighbourhoods to the completely-coloured graph by colouring the edges in between the neighbourhoods (see Figure~\ref{fig:glueing}). This method has been used many times to compute upper bounds on Ramsey numbers, for example for \( R(K_4,K_5) \)~\cite{R45}. 

To use this glueing strategy for \( \ram{5}{6} \) we need appropriate collections of \( (J_4,J_6) \)-graphs and \( (J_5,J_5) \)-graphs. More specifically for \(\RamSet(J_5,J_6;37)\) we have \( n_1+n_2 = 36 \), but also \( n_1<17 = R(J_4,J_6) \)~\cite{mcnamara1991ramsey} and \( n_2<22=R(J_5,J_5) \)~\cite{clapham1989ramsey}.  Therefore we only need the sets with \( n_1 \geq 15 \) and \( n_2 \geq 20 \).
The first set was fully computed by Radziszowski~\cite{sprCom}. The second set was computed for orders 20 and 21 in~\cite{SRK5e}.

We independently verified the correctness of the first set of graphs by writing a plugin for the generator \verb|geng|~\cite{nauty-website, mckay_14} to generate $(J_4,J_6)$-graphs.
For the mentioned \( (J_5,J_5; n) \)-graphs, we wrote a gluing algorithm based on the techniques explained below, and generated all Ramsey graphs for $19\leq n \leq 21$.
 
The results can be found in Table~\ref{table:countJ4J6} and Table~\ref{table:countJ5J5} in the Appendix.

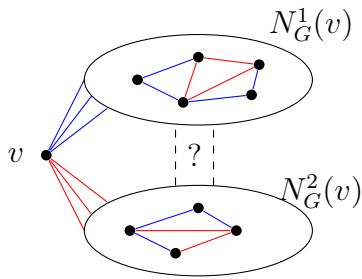
\begin{figure}
	\centering
	\begin{tikzpicture}
	%\path (0,0) rectangle (4,4);
	
	\path[every node/.append style={circle, fill=black, minimum size=4pt, label distance=0pt, inner sep=0pt}]
	(0.5,0) node[label={[label distance=5pt]180:\(v\)}] (v) {}
	(1.7,1) node (0) {}
	(3.2,0.8) node (1) {}
	(2.5,1.3) node (2) {}
	(3.3,1.2) node (3) {}
	(2.3,0.7) node (4) {}
	
	(2.5,-0.7) node (A) {}
	(1.6,-1) node (B) {}
	(3,-1) node (C) {}
	(2.2,-1.3) node (D) {};
	
	\draw (3) edge[red] (2) edge[red] (4);
	\draw (2) edge[red] (4);
	\draw (0) edge[blue] (2) edge[blue] (4);
	\draw (1) edge[blue] (4) edge[blue] (3);
	\draw (A) edge[blue] (B) edge[blue] (C);
	\draw (B) edge[red] (C) edge[blue] (D);
	\draw (D) edge[red] (C);
	
	\draw (v) -- (1,1) [blue];
	\draw (v) -- (1.1,0.79) [blue];
	\draw (v) -- (1.3,0.62) [blue];
	\draw (v) -- (1,-1) [red];
	\draw (v) -- (1.1,-0.79) [red];
	\draw (v) -- (1.3,-0.62) [red];
	\draw (2.2,-0.4) -- (2.2,0.4) [dashed];
	\draw (2.7,-0.4) -- (2.7,0.4) [dashed];
	\draw (2.45, 0) node {?};
	
	\draw[] (2.5, 1) ellipse (1.5 and 0.6) node[label={[label distance=20pt]25:$N_G^1(v)$}] {};
	\draw[] (2.5, -1) ellipse (1.5 and 0.6)node[label={[label distance=22pt]10:$N_G^2(v)$}] {};
	\end{tikzpicture}
	\caption{Illustration of the glueing approach.}
	\label{fig:glueing}
\end{figure}

\subsection{Triangle constraints}
\label{subsect:UB_tf}

To prove the upper bound \( \ram{5}{6} \leq 37 \), we need to consider every pair of graphs \( G_1 \in \RamSet(J_4,J_6;n_1) \) and \( G_2 \in \RamSet(J_5,J_5;n_2) \) (with $n_1+n_2=36$), and use them as \( \neigh{G}{1}{v} \) and \( \neigh{G}{2}{v} \) respectively. In this way we obtain a graph which is completely coloured except for the edges between \( \neigh{G}{1}{v} \) and \( \neigh{G}{2}{v} \). If we can show that for every such pair it is impossible to colour these remaining edges without creating a \( J_5 \) in the first colour or a \( J_6 \) in the second, we have proved the theorem.    

Not all of these pairs $(G_1, G_2)$ have to be considered. By counting the number of monochromatic triangles in the hypothetical Ramsey graphs resulting from such a gluing, we can eliminate certain pairs.

Let $G$ be a two-coloured $K_n$. If the degrees of the vertices in each colour are known, the number of monochromatic triangles \( T(G) \) in $G$ can be computed using the following theorem by Goodman.

\begin{lemma}[Goodman~\cite{Goodman}] Let $G$ be a two-coloured $K_n$. Then
	\label{thm:Goodman}
	\[
	T(G) =\binom{n}{3} - \frac{1}{2}\sum_{i = 0}^{n-1}[\bar n_i\cdot i \cdot (n-1-i)]
	\]
\end{lemma}
\noindent
On the other hand, the number of monochromatic triangles in $G$ can also be counted by considering the number of edges coloured with colour $i$ (for $i \in \{1,2\}$) in the neighbourhood $\neigh{G}{i}{v}$ of each vertex~$v$.

\begin{lemma}
	\label{thm:triangles}
	\[
	T(G) = \frac{1}{3}\sum_{v \in V}{[ e_1(\neigh{G}{1}{v}) + e_2(\neigh{G}{2}{v})] }
	\]
\end{lemma}
\begin{proof}
	Each edge with colour $i$ (for $i \in \{1,2\}$) in \( \neigh{G}{i}{v} \) extends to a monochromatic triangle with colour $i$ when combined with $v$. Summing over all vertices counts each triangle three times. 
\end{proof}

From the definitions it is clear that if $G = (V, E)$ is a \( (J_k,J_l;n) \)-graph, then \( \forall \, v \in V: \ e_1(\neigh{G}{1}{v})\leq E_1(J_{k-1},J_l;\deg_1(v)) \), and similarly for the second colour.
For a \( (J_k,J_l) \)-graph, the difference in number of edges from the ``extremal'' case is expressed in the following \textit{deficiency function}:
\[
\delta_G(v) :=  E_1(J_{k-1},J_l;\deg_1(v)) - e_{1}(\neigh{G}{1}{v}) + E_2(J_k,J_{l-1};\deg_2(v)) - e_2(\neigh{G}{2}{v})
\] 
It follows for a \( (J_k,J_l) \)-graph $G$ that \(\forall \, v \in V: \delta_G(v)\geq 0\) and therefore also that \( \sum_{v \in V}{\delta_G(v)}\geq 0 \).  

Combining the previous results leads to (see~\cite{SRK5e} for details):
\begin{equation}\label{eq_sum_delta}
%\small
\sum_{v \in V}{\delta_G(v)} 
= - 3\binom{n}{3} + \sum_{i = 0}^{n-1}
{ \bar n_i [E_1(J_{k-1},J_{l};i) + E_2(J_k,J_{l-1};n-1-i) + \frac{3i\cdot(n-1-i)}{2}] }
\end{equation}
Here $n$ denotes the order of $G$. This gives us the sum of the deficiencies as a function of only the degree sequence of $G$ (given that we know the values of $E_1$ and $E_2$). If this sum is negative, then such a Ramsey graph cannot exist.

%\medskip

\begin{theorem}
	\label{thm:ub_r56}
	\( \ram{5}{6} \leq 37 \).
\end{theorem}
\begin{proof}
	We use the framework outlined above.
	From the counts in Table~\ref{table:countJ4J6} and Table~\ref{table:countJ5J5} it follows that: 
	
	\begin{tabular}{l l}
		\( E_1(J_4,J_6;15)=45 \), & \( E_1(J_4,J_6;16)=50 \), \\
		\( E_2(J_5,J_5;20)=100 \), & \( E_2(J_5,J_5;21)=105 \). \\
	\end{tabular}

	\noindent
	Let $G = (V,E)$ be a \( (J_5,J_6;37) \)-graph, then -- as already mentioned in Section~\ref{subsect:UB_prelim} -- every vertex of $G$ has either 15 or 16 neighbours in the first colour.
	Equation~(\ref{eq_sum_delta}) now leads to:
	\begin{align*}
	\sum_{v \in V}{\delta_G(v)} 
	&= - 23310 +  \bar n_{15}\cdot (45 + 105 + 472.5) +  \bar n_{16}\cdot (50 + 100 + 480) \\
	&= - 23310 + \bar n_{15}\cdot 622.5 + \bar n_{16} \cdot 630
	\end{align*}
	Under the given constraints, this sum is maximal for $\bar n_{16}=37$ and gives \( \sum_{v \in V}{\delta_G(v)} = 0 \); all other combinations lead to a negative sum.
	
	Therefore the only remaining possibility for a \( (J_5,J_6;37) \)-graph $G$ is one where every vertex $v$ has deficiency $0$, and 16 neighbours in the first colour.
	 Hence for each $v \in V$, \( \neigh{G}{1}{v} \) contains exactly 50 edges with colour~1. Since every of those edges leads to a monochromatic triangle with colour~1 in $G$ and we count every triangle three times, there must be exactly \( \frac{37\cdot50}{3} \) monochromatic triangles with colour~1 in $G$. However, this is not an integer, hence $G$ does not exist.
\end{proof}

With this new upper bound, the classical inequality on Ramsey numbers yields the following Corollary.

\begin{corollary}\label{corr:j5j7}
$R(J_5,J_7) \leq R(J_4,J_7) + R(J_5,J_6) \leq 28 + 37 = 65.$
\end{corollary}

It should be noted that this upper bound $R(J_5,J_7) \leq 65$ was recently already established in~\cite{SemiDef} by Lidick{\'y} and Pfender using computational techniques. In the next section we will show that the upper bounds $R(J_5,J_6) \leq 37$ and $R(J_5,J_7) \leq 65$ are tight.

\section{Improving lower bounds on Ramsey numbers}
\label{sect:LB}

To establish a lower bound on a Ramsey number, it suffices to construct a single Ramsey graph. In this section we try to find new lower bounds based on circulant and block-circulant graphs, and we do this in an exhaustive way.

\subsection{Circulant graphs}
\label{subsect:circ}

A graph is called \textit{circulant} if it is of the form \( G=(V,E) \) where \( V=\{0,\dots,n-1\} \) and \( \{i,j\}\in E \Leftrightarrow (i-j)\pmod n \in D\) for some \( D \subseteq \{1,\dots,n-1\} \), which is closed under additive inverses modulo $n$.
The adjacency matrix of such a graph has the property that every row can be obtained by rotating the preceding row by one position (always in the same direction). This is also called a \textit{circulant matrix}. The set $D$ is called the \textit{generating set} of this matrix. To see this as a Ramsey graph, let all edges of $G$ be blue and those in the complementary graph (which is also circulant) be red, which yields a two-coloured complete graph. 
For classical Ramsey numbers, \( R(K_3,K_3) \), \( R(K_3,K_4) \), \( R(K_3,K_5) \), \( R(K_3,K_9) \), \( R(K_4,K_4) \), and \( R(K_4,K_5) \) all have extremal graphs which are circulant. In some cases these graphs are even unique as extremal graphs, e.g.\ for $R(K_3,K_9)$~\cite{staszek13}.
%However, for other Ramsey numbers such as \( R(K_4,K_6) \), the largest-known Ramsey graphs seem to have no apparent extra structure. 

In this work, we designed and implemented an algorithm which enumerates all circulant Ramsey graphs for various parameters.
It is based on a backtracking algorithm and exploits circularity to speed up the search. For example, to determine if a certain clique is present in the graph, one can limit this to cliques containing certain ``canonical'' edges (see~\cite{thesissteven} for more details). The structure in the adjacency matrix of these graphs also allows to perform bitwise operations to accelerate the search.

No improvements on lower bounds of classical Ramsey numbers were found, agreeing with what was reported in~\cite{HaKr}. Kuznetsov~\cite{Kuznetsov} also calculated the best-possible bounds for certain classical Ramsey numbers within the class of \textit{distance graphs}, a generalisation of circulant graphs.

Our algorithm was extended to generate circulant graphs for \( R(H_1,H_2) \), where \( H_i \) is one of the following graphs:
\( K_n \), \( J_n \), \( C_n \), \( W_n \), \( K_{n,m} \). 
Here \(C_n\) denotes a cycle of length $n$, \(W_n\) is a wheel graph on $n$ vertices (i.e.\ a graph obtained by connecting a single vertex to all vertices of a $C_{n-1}$), and $K_{n,m}$ is the complete bipartite graph with partite sets of orders $n$ and $m$.
For these cases, several lower bounds could be improved (shown in Table~\ref{tbl:dropEdge} and Table~\ref{tbl:otherLower} from Section~\ref{subsect:results}), but most of them were later further improved using block-circulant graphs (see Section~\ref{subsect:blockcirc}).
Our algorithm is also suitable to search for multi-colour Ramsey graphs.

\begin{claim}
None of the lower bounds for Ramsey numbers of the form \( R(H_1,H_2) \) with $H_i \in \{ K_n, J_n, C_n, W_n, K_{n,m} \}$ (for $i \in \{1,2\}$) reported in Table~\ref{tbl:dropEdge} or Table~\ref{tbl:otherLower}, or mentioned in any table in~\cite{dynSur} can be improved using circulant Ramsey graphs on 64 or fewer vertices.
\end{claim}

\subsection{Block-circulant graphs}
\label{subsect:blockcirc}
The structure of circulant graphs turns out to be too restrictive to improve challenging lower bounds on small Ramsey numbers. 
Therefore we also considered a natural generalisation of circulant graphs: block-circulant graphs. 
These are graphs of which the adjacency matrix is composed of equally-sized circulant matrices:

\[ A = 
\begin{bmatrix}
C_{11} & C_{12} & \dots & C_{1k}\\
C_{21} & C_{22} & \dots & C_{2k}\\
\vdots & \vdots & & \vdots\\
C_{k1} & C_{k2} & \dots & C_{kk}
\end{bmatrix}
\]

For this adjacency matrix to represent a simple graph, it is necessary that \( C_{i,j}=C_{j,i}^T \) for all \( i, j \). This is possible because the transposition of a circulant matrix is also circulant. If a graph $G$ has an adjacency matrix of the form of $A$, we say that it is a block-circulant graph on $k$ blocks. It is uniquely determined by giving the generating set (i.e.\ the first row) of each block in the upper triangle of $A$. Let \(D_{i,j}\) be this generating set of \(C_{i,j}\).

An example of a block-circulant matrix on 3 blocks would be
\[ A_1 = 
\begin{bmatrix}
(3,4,5,6) & (0,1,2)  & (0,2,4)\\
(0,7,8) & (2,3,6,7) & (0,4,8)\\
(0,5,7) & (0, 1, 5) & (1,3,6,8)
\end{bmatrix}_{9}
\text{,}
\]
where the brackets denote the generating sets, and the subscript indicates that we are working modulo 9. This represents the adjacency matrix of $O_6^-(2)$, the unique extremal graph for $R(J_4,J_7)$~\cite{mcnamara1991ramsey}.

These block-circulant graphs have been used before in search of lower bounds on Ramsey numbers. For example in~\cite{exoo1998some, Exoo2000} heuristic searches were performed for block-circulant Ramsey graphs for Ramsey numbers of the form \(R(K_k, K_l)\) and \(\ram{k}{l}\). They were also used to generate starting points for local search algorithms~\cite{ExT}.
However, we will follow an exhaustive approach and we will see that in some cases these earlier heuristic searches missed the best-possible values.

\paragraph*{}
The basic idea behind our exhaustive generation algorithm for block-circulant graphs is similar to the idea behind the algorithm for circulant graphs from Section~\ref{subsect:circ}: perform a backtracking search over the free parameters of the graph, in this case the generating sets of the circulant matrices in the upper triangle of the adjacency matrix.
Pruning is done whenever a forbidden subgraph $H_1$ or $H_2$ is formed (when searching for Ramsey graphs for $R(H_1,H_2)$), and the search for these subgraphs is restricted by a representative of the newly-coloured edges.

Extra care was taken to avoid generating isomorphic (partially-coloured) graphs.  As will be shown in the next paragraphs, many isomorphisms can be detected directly from the structure of the block-circulant graphs. We did not eliminate further sporadic isomorphisms.

In the following, let \( G \) be a block-circulant graph of order $n$ on $k$ blocks with an adjacency matrix as depicted at the beginning of this subsection.
\begin{lemma}
\label{lem:perm}
Let \( \pi: \range{k}\rightarrow \range{k} \) be a permutation.
Then the graph \( G' \) with adjacency matrix \( (C_{\pi(i),\pi(j)})_{i,j\in\range{k}} \) is also block-circulant and is isomorphic to $G$.
\end{lemma}
%\noindent
Most isomorphisms of this kind can be avoided by defining an ordering on every possible generating set and accepting an adjacency matrix only if the generating sets of the blocks on the diagonal are in non-decreasing order. That is: \( \forall \, i<j: D_{i,i}\leq D_{j,j} \), where ``$\leq$'' denotes the chosen ordering. Only when some blocks on the diagonal are equal, these isomorphisms can remain undetected.

Another form of structural isomorphism in block-circulant graphs originates from \\ ``rotating'' one block relative to the other blocks:
\begin{lemma}
\label{lem:rot}
	Let \( d \in \range{k}  \) and \( r \in \N \). Then the graph \( G' \) constructed from \( G \) by rotating every \( C_{i,d} \) $r$ times to the right, and \( C_{d,i} \) $r$ times to the left \( (i \neq d) \), is also block-circulant and is isomorphic to $G$ (where rotating $C_{i,d}$ means cyclically rotating each row of that submatrix). 
\end{lemma}

To avoid generating graphs which are isomorphic because of this reason, we fix a certain rotation of each block. For this it is mostly sufficient to demand that $C_{1,d}$, the first block of each column, is generated by a \textit{Lyndon word}, i.e.\ a bitstring which is the lexicographically smallest among all of its circular rotations. If there are multiple rotations of $C_{1,d}$ giving the same Lyndon word, then isomorphisms might still occur because of the other blocks $C_{d',d}$. These ties can then be broken by requiring $C_{2,d}$ to be lexicographically smallest among all rotations that fix $C_{1,d}$ etc.

\begin{lemma}
\label{lem:mult}
	If $q$ is co-prime with $n/k$ (i.e. $q\in \mathbb{Z}_{n/k}^{*}$), then applying \( D\mapsto q\cdot D :=\{q\cdot d \pmod {n/k} \ | \ d\in D\} \) to each \(D_{i,j}\) leads to a graph which is isomorphic to $G$.
\end{lemma}
We can avoid nearly all isomorphisms of this kind by only proceeding the search if the sequence $(D_{1,1},\dots,D_{k,k})$ is the lexicographically smallest among all multiples \( (q\cdot D_{1,1},\dots,q\cdot D_{k,k}) \), with $q$ co-prime with $n/k$ (again using the ordering defined in Lemma~\ref{lem:perm}). Note that this criterion can already be checked in partially-filled matrices if it is used in combination with the criterion of Lemma~\ref{lem:perm}. Also notice that, since all blocks on the diagonal are symmetric, multiplication with $-1\in \mathbb{Z}_{n/k}^{*}$, will fix all blocks on the diagonal. Therefore, an extra condition can be added (looking at non-diagonal blocks), to decide whether we will accept this graph or its multiplication with $-1$.

\begin{theorem}
	\label{thm:canon}
	For every block-circulant graph $G$, there exists at least one block-circulant graph $G'\cong G$ that meets all of the above criteria.
\end{theorem}
\begin{proof}
	Starting from $G$, we preform the following operations consecutively to make the labelling ``canonical'':
	\begin{itemize}
		\item Compute a $q \in \mathbb{Z}_{n/k}^{*}$ for which $\{q\cdot D_{i,i}\ | \ {1\leq i \leq k}\}$ is lexicographically minimal (seen as a multiset). Multiply all blocks with this $q$.
		\item Apply a permutation of the blocks, such that the blocks on the diagonal are in non-decreasing order. I.e.: sort the diagonal blocks.
		\item Rotate each column until all $C_{1,i}$ are generated by Lyndon words, $1<i\leq k$. If there are multiple rotations that minimise $C_{1,i}$, choose the one among them that minimises $C_{2,i}$, and so on.
		\item If $D_{1,2}$ is now bigger than the Lyndon rotation of $(-1)\cdot D_{1,2}$, multiply all blocks by $-1$ and repeat step 3.
	\end{itemize}
The resulting graph $G'$ is isomorphic to $G$ and is still block-circulant.
\end{proof}

To illustrate this process, note that the adjacency matrix $A_1$ depicted at the beginning of this subsection would pass the tests from Lemma~\ref{lem:rot} and Lemma~\ref{lem:mult}, but the generating sets on the diagonal are not in increasing order. Applying the operations described in Theorem~\ref{thm:canon} results in the following adjacency matrix, which is also the only one generated by our program for $R(J_4,J_7;28)$:
\[ A_1^* = 
\begin{bmatrix}
(1,3,6,8) & (0,1,5)  & (0,2,4)\\
(0,4,8) & (2,3,6,7) & (0,1,2)\\
(0,5,7) & (0,7,8) & (3,4,5,6)
\end{bmatrix}_{9}
\]

We now give counts of some concrete cases to give an indication of how many isomorphic graphs are avoided by the generator and how many remain. There are 32076 block-circulant Ramsey-$(J_4,J_8;27)$-graphs on three blocks. Of those, only 17 are non-isomorphic. With all of the above restrictions, our program generated 44 graphs.
Of the 26 block-circulant Ramsey-$(K_4,K_8;54)$-graphs generated on 3 blocks, 23 are non-isomorphic.
For block-circulant Ramsey-$(K_4,J_7;36)$-graphs on 4 blocks, the program outputs 2 graphs, which are non-isomorphic. 

\paragraph*{}
The order in which the edges are coloured was also taken into consideration. To find illegal subgraphs as soon as possible, we opted to build the partial graphs in such a way that the already-coloured edges are only between a small number of vertices.
In general, this gives a higher probability of creating cliques. We achieve this by filling in the adjacency matrix column by column, from left to right. We refer to~\cite{thesissteven} for more details on the algorithm. The source code of our implementation of this algorithm can be obtained from \url{https://github.com/Steven-VO/circulant-Ramsey}

\subsection{Other graphs}
\label{subsect:other}
In search of new lower bounds, we also tested all vertex-transitive graphs up to 47 vertices for their Ramsey properties concerning \( J_k \) and \( K_k \). This set of graphs was computed by Holt and Royle~\cite{holt2020census}.
Some improved bounds were found on Ramsey numbers of the form \( \ram{k}{l} \) and \( R(K_k,J_l) \), but all of them could also be reached (or even improved) using block-circulant graphs, so we do not list them separately here.

Many known interesting strongly-regular graphs where also checked. These are regular graphs where every two adjacent vertices share the same number of common neighbours, and the same is true for non-adjacent vertices. For some parameter sets, all strongly-regular graphs have been enumerated (see e.g.~\cite{coolsaet2006strongly}). Other sporadic interesting cases are described in~\cite{StrongReg}. This led to the discovery of \( VO_{6}^{-}(2) \) as a \( (J_5,J_7) \)-graph, which improved the previous lower bound of $\ram{5}{7}$~\cite{Exoo2000} by 25 to 65. (See the next section for details).

\subsection{Results}
\label{subsect:results}

With the techniques from Section~\ref{sect:LB}, no lower bounds on classical Ramsey numbers were improved.
However, several new lower bounds where found for Ramsey numbers of the form \( \ram{k}{l} \) and \( R(K_k,J_l) \), including two exact values: \( \ram{5}{6}=37 \) and \( \ram{5}{7}=65 \). These results are presented in Table~\ref{tbl:dropEdge}.

The Ramsey graphs which establish the new lower bounds (and the source code of our  algorithms to generate circulant and block-circulant Ramsey graphs) can be obtained from \url{https://github.com/Steven-VO/circulant-Ramsey} as well as from the \textit{House of Graphs}~\cite{hog} through the links in Table~\ref{tbl:dropEdge}. In each case we computationally verified that these graphs are indeed Ramsey graphs for the given parameters using two independent programs (see Section~\ref{subsect:testing} for details).

Many best-known lower bounds could be reproduced within seconds of CPU time using our generators for circulant and block-circulant Ramsey graphs. The complexity of the exhaustive non-existence results for block-circulant graphs grows very rapidly with increasing parameters, and we therefore limited such searches to about 3 days of CPU-time for each case.
Sometimes the largest found block-circulant Ramsey graph could be extended by one extra vertex, connected in a specific way to the other vertices. This is denoted by ``+1'' in Table~\ref{tbl:otherLower}.
In some cases we were able to construct larger Ramsey graphs by performing a \textit{local search} on a block-circulant Ramsey graph, that is: we remove some vertices and then add more new vertices in all possible ways and check if any of them is still a Ramsey graph. This is denoted by ``LS'' in Table~\ref{tbl:dropEdge}.

%\paragraph*{}
We also found an interesting link between three different Ramsey numbers. McNamara and Radziszowski~\cite{mcnamara1991ramsey} showed that there is a unique extremal graph for \( \ram{4}{7} \), known as the complement of the Schl\"afli graph. This graph can also be constructed as \( O^{-}_{6}(2) \): the orthogonal points on an elliptic quadric in \( PG(5,2) \) (see~\cite{StrongReg} for details).
 \( NO^{-}_6(2) \) is a geometrically related graph, which we found with the block-circulant generator, and which we proved to be extremal as a \( (J_5,J_6) \)-graph. These  graphs are combined in \(  VO_{6}^{-}(2) \), which turned out to be an extremal Ramsey graph for \( \ram{5}{7} \). These three graphs are all vertex-transitive, strongly-regular and block-circulant.

\begin{theorem}
\( \ram{5}{6}=37 \).
\end{theorem} 
 \begin{proof}
 The upper bound follows from Theorem~\ref{thm:ub_r56}. The lower bound is established by \( NO^{-}_6(2) \) for which we computationally verified that it is a $(J_5,J_6;36)$-graph using two independent programs.
 \end{proof}
 
 \begin{theorem}
\( \ram{5}{7}=65 \) and the graph \(  VO_{6}^{-}(2) \) is the only extremal Ramsey graph for $\ram{5}{7}$.
\end{theorem} 
\begin{proof}
The upper bound follows from~\cite{SemiDef} or Corollary~\ref{corr:j5j7}.
The lower bound is established by \(  VO_{6}^{-}(2) \) for which we computationally verified that it is a $(J_5,J_7;64)$-graph using two independent programs.

Since \(O^{-}_{6}(2)  \) is unique as extremal \((J_4,J_7)  \)-graph~\cite{mcnamara1991ramsey}, and \( \ram{5}{7} = \ram{5}{6} + \ram{4}{7}\), it follows from the simple arguments from Section~\ref{sect:UB} that every \( (J_5,J_7;64) \)-graph $G$ must have the following property: \( \forall \, v \in V: \neigh{G}{1}{v} \cong O^{-}_{6}(2) \).  The graphs with this property have been completely characterised (there are only two of them which are connected)~\cite{buekenhout}. The first is \(  VO_{6}^{-}(2) \); and the other graph with this property is known as \(TO_{6}^{-}(2) \), but has independence number~$7$.
Therefore  \(  VO_{6}^{-}(2) \) is the only extremal Ramsey graph for $\ram{5}{7}$. 
\end{proof}

\begin{table}[htb!]
\centering
\setlength{\tabcolsep}{4pt} %6 standard
\small
	\begin{tabular}{l|c|c|c|c}
		Old bounds & New LB & Method & Old LB reference & HoG id\\ \hline
		$ 47 \leq R(K_3,J_{12})\leq 53 $  & 49 & Block-circulant 
		& Implied by $R(K_3,K_{11})$ & \href{https://hog.grinvin.org/ViewGraphInfo.action?id=44120}{44120}\\
		$ 60\leq R(K_3,J_{14}) \leq 71 $  & 61 & Block-circulant 
		& Implied by $R(K_3,K_{13})$ & \href{https://hog.grinvin.org/ViewGraphInfo.action?id=44122}{44122}\\
		\( 29\leq R(J_4,K_8) \leq 39\)    & 36 & Block-circulant  
		& Implied by $R(J_4,K_7)$   & \href{https://hog.grinvin.org/ViewGraphInfo.action?id=44116}{44116}\\
		$ 36\leq R(J_4,K_9)\leq 56 $ & 41& Block-circulant
		 & Implied by $R(K_3,K_9)$ & \href{https://hog.grinvin.org/ViewGraphInfo.action?id=44132}{44132} \\
		 $ 41\leq R(J_4,J_{10})\leq 63 $ & 43& Block-circulant
		 & Exoo (2000)~\cite{Exoo2000} & \href{https://hog.grinvin.org/ViewGraphInfo.action?id=45620}{45620} \\
		\( 41 \leq R(J_4,K_{10}) \leq 65 \)   & 49  & Block-circulant
		& Implied by $R(J_4,J_{10}$)  & \href{https://hog.grinvin.org/ViewGraphInfo.action?id=44124}{44124}\\
		\( 74 \leq R(J_4, J_{16})  \)  & 82 & Strongly regular graph
		& Implied by \( R(K_3,K_{15}) \) & \href{https://hog.grinvin.org/ViewGraphInfo.action?id=962}{962} \\
		\(49 \leq R(K_4,J_8)\leq 74 \)    & 52  & Block-circulant
		& Implied by $R(K_4,K_7)$ & \href{https://hog.grinvin.org/ViewGraphInfo.action?id=44112}{44112}\\
		\(59 \leq R(K_4,J_9)\leq 105 \)    & 62  & Circulant
		& Implied by $R(K_4,K_8)$ & \href{https://hog.grinvin.org/ViewGraphInfo.action?id=44130}{44130}\\
		\( 31 \leq R(J_5,J_6) \leq 38 \)   & \textbf{37} & Block-circulant
		&  Exoo (2000)~\cite{Exoo2000}  & \href{https://hog.grinvin.org/ViewGraphInfo.action?id=34470}{34470}\\
		\( 37 \leq R(J_5,K_6)\leq 53 \)  & 43 & Block-circulant (+LS) 
		& Exoo (2000)~\cite{Exoo2000} & \href{https://hog.grinvin.org/ViewGraphInfo.action?id=44118}{44118}\\
		\( 40 \leq R(J_5,J_7)\leq 65 \)  & \textbf{65} & Strongly regular graph
		& Exoo (2000)~\cite{Exoo2000} & \href{https://hog.grinvin.org/ViewGraphInfo.action?id=35441}{35441}\\
		\( 80 \leq R(K_5,J_8) \leq 175  \)  & 81 & Circulant
		& Implied by \( R(K_5,K_7) \) & \href{https://hog.grinvin.org/ViewGraphInfo.action?id=45622}{45622}\\
		\( 101 \leq R(K_5, J_9) \leq 275  \)  & 121  & Strongly regular graph
		& Implied by \( R(K_5,K_8) \) & \href{https://hog.grinvin.org/ViewGraphInfo.action?id=44126}{44126}\\
		\( 80 \leq R(J_6,J_8) \leq 218  \)  & 83 & Circulant
		 & Implied by \( R(K_5,K_7) \) & \href{https://hog.grinvin.org/ViewGraphInfo.action?id=44128}{44128}\\
	\end{tabular}
	\caption{Improved lower bounds for Ramsey numbers of the form \( R(J_k,J_l) \), together with the best-known bounds prior to this work. Exact values are marked in bold. ``LS'' stands for local search. The last column refers to the ids of these new Ramsey graphs in the \textit{House of Graphs}~\cite{hog}.}
	\label{tbl:dropEdge}
\end{table}

We focussed on Ramsey numbers involving \( J_k \), but also executed our algorithms on other combinations of parameters, including multi-colour Ramsey numbers. We obtained several improvements over the current lower bounds, including some exact values for Ramsey numbers on wheels and complete bipartite graphs. These results are shown in Table~\ref{tbl:otherLower}. Note that \(W_n\) (i.e.\ a wheel graph on $n$ vertices) is not contained in \(W_{n+1}\). Therefore these Ramsey numbers are not necessarily increasing in $n$. We believe that more improvements could be possible by applying the same techniques to Ramsey numbers with different sets of parameters. But as there is a very large number of parameter combinations, we focussed on the most common cases.

	\begin{table}[htb!]
		\centering
		\begin{tabular}{l|c|c|c}
			Old bounds & New LB  & Method & HoG id\\ \hline
			\( R(W_7,W_4) \leq 21 \) & \textbf{21}  & Block-circulant & \href{https://hog.grinvin.org/ViewGraphInfo.action?id=35445}{35445}\\ 
			\( R(W_7,W_7)\leq 19 \)     & \textbf{19}  & Block-circulant & \href{https://hog.grinvin.org/ViewGraphInfo.action?id=35443}{35443} \\
			\(  R(W_9,W_9) \)     &  21   & Block-circulant& \href{https://hog.grinvin.org/ViewGraphInfo.action?id=44138}{44138} \\
			\(  R(K_6,W_6) \leq 40 \)     &  34   & Block-circulant & \href{https://hog.grinvin.org/ViewGraphInfo.action?id=44134}{44134} \\
			\( 43 \leq R(K_7, W_5)\leq 50 \) & 45   & Block-circulant & \href{https://hog.grinvin.org/ViewGraphInfo.action?id=44136}{44136}  \\
				\(  R(K_7,W_6) \leq 55 \)     &   45  & Block-circulant & \href{https://hog.grinvin.org/ViewGraphInfo.action?id=44136}{44136}\\
			\( 39 \leq R(K_{11},C_4)\leq 44 \) & 40 & Block-circulant & \href{https://hog.grinvin.org/ViewGraphInfo.action?id=44167}{44167}\\ 
			24 \(\leq  R(K_{2,6},K_{2,8}) \leq 25 \) & \textbf{25} & Block-circulant & \href{https://hog.grinvin.org/ViewGraphInfo.action?id=44140}{44140}\\ 
			28 \(\leq  R(K_{2,7},K_{2,10}) \leq 31 \) & 29 & Block-circulant & \href{https://hog.grinvin.org/ViewGraphInfo.action?id=44142}{44142}\\ 
			32 \(\leq  R(K_{2,8},K_{2,10}) \leq 33 \) & \textbf{33} & Block-circulant & \href{https://hog.grinvin.org/ViewGraphInfo.action?id=44144}{44144} \\ 
			\(R(K_{2,8},K_{2,11}) \leq 35 \) & \textbf{35} & Block-circulant & \href{https://hog.grinvin.org/ViewGraphInfo.action?id=44146}{44146}\\ 
			\( 36\leq R(K_{2,9},K_{2,11}) \leq 37 \) & \textbf{37} & Block-circulant & \href{https://hog.grinvin.org/ViewGraphInfo.action?id=44148}{44148}\\ 
			\( R(K_{3,4},K_{2,5}) \leq 20 \) & \textbf{20} & Block-circulant (+1) & \href{https://hog.grinvin.org/ViewGraphInfo.action?id=44150}{44150}\\ 
			\( R(K_{3,4},K_{3,3}) \leq 20 \) & 19 & Block-circulant & \href{https://hog.grinvin.org/ViewGraphInfo.action?id=44152}{44152}\\ 
			\( R(K_{3,4},K_{3,4}) \leq 25 \) & \textbf{25} & Block-circulant & \href{https://hog.grinvin.org/ViewGraphInfo.action?id=44154}{44154}\\
			\( R(K_{3,5},K_{2,4}) \leq 20 \) & 19 & Block-circulant & \href{https://hog.grinvin.org/ViewGraphInfo.action?id=45601}{45601} \\
			\( R(K_{3,5},K_{2,5}) \leq 23 \) & 21 & Circulant & \href{https://hog.grinvin.org/ViewGraphInfo.action?id=44156}{44156} \\
			\( R(K_{3,5},K_{3,3}) \leq 24 \) & 21 & Circulant & \href{https://hog.grinvin.org/ViewGraphInfo.action?id=44156}{44156}\\
			\( R(K_{3,5},K_{3,4}) \leq 29 \) & 25 & Block-circulant & \href{https://hog.grinvin.org/ViewGraphInfo.action?id=44158}{44158}\\  
			\(30 \leq R(K_{3,5},K_{3,5}) \leq 33 \) & \textbf{33} & Block-circulant & \href{https://hog.grinvin.org/ViewGraphInfo.action?id=44160}{44160}\\ 
			\(30 \leq R(K_{4,4},K_{4,4}) \leq 49 \) & 33 & Block-circulant & \href{https://hog.grinvin.org/ViewGraphInfo.action?id=44162}{44162}\\ 
			\( 30 \leq R(K_3,J_4,K_4)\leq 40 \)   & 31 & Block-circulant & (\href{https://github.com/Steven-VO/circulant-Ramsey}{GitHub})\\
			\( 28 \leq R(K_4,J_4,C_4)\leq 36 \)  & 29  & Block-circulant & (\href{https://github.com/Steven-VO/circulant-Ramsey}{GitHub})\\
		\end{tabular}
		\caption{Improved lower bounds for various Ramsey numbers. Exact values are marked in bold. ``+1'' denotes a one-vertex extension. The upper bounds are all from~\cite{SemiDef} or~\cite{dynSur}. The last column refers to the ids of these new Ramsey graphs in the \textit{House of Graphs}~\cite{hog}. The three-coloured graphs are only available on \href{https://github.com/Steven-VO/circulant-Ramsey}{GitHub}.}
		\label{tbl:otherLower}
	\end{table}

\subsection{Correctness testing}
\label{subsect:testing}
All programs were written in the programming language C. The counts of the generated \( (J_5,J_5) \)-graphs on 20 and 21 vertices (cf.\ Table~\ref{table:countJ5J5}) are in complete agreement with previous results in~\cite{SRK5e}. As a partial verification for the correctness of \( \RamSet(J_5,J_5;19) \), up to 2 vertices were removed and added again in every possible way for every graph in \( \RamSet(J_5,J_5;19) \). This led to exactly the same set of Ramsey graphs.
The correctness of the counts in Table~\ref{table:countJ4J6} was tested by writing a plugin for the program \verb|geng|~\cite{nauty-website, mckay_14}. This yielded exactly the same graphs as those we received from Radziszowski~\cite{sprCom}.

The graphs witnessing the lower bounds reported in Table~\ref{tbl:dropEdge} and Table~\ref{tbl:otherLower}  were all independently verified using the Graph-package in \textit{SageMath}.
Together with the source code of our generators, they can be obtained from \url{https://github.com/Steven-VO/circulant-Ramsey}

\section{Further research}
\label{sect:further_research}

\begin{problem}
Is \( NO^{-}_6(2) \) the only extremal Ramsey graph for $\ram{5}{6}$?
\end{problem}

We strongly suspect that this is the case as the closely related extremal Ramsey graphs for \( \ram{4}{7} \) and \( \ram{5}{7} \) are unique as well. The following computational evidence also seems to indicate that \(NO^{-}_6(2) \)  is the only $(J_5,J_6;36)$-graph:

\begin{itemize}
\item A local search was performed on $NO^{-}_6(2)$: we removed 2 vertices and then readded them in all possible ways and this did not yield any additional $(J_5,J_6;36)$-graphs. So within a distance of 2 no other $(J_5,J_6;36)$-graphs exist.
\item Up to isomorphism $NO^{-}_6(2)$ is the only block-circulant \( (J_5,J_6;36) \)-graph on 6 blocks or less.
\end{itemize}

It can be observed that Ramsey numbers of the form \( \ram{k}{l} \) seem to ``behave better'' than the Ramsey numbers \( R(K_k,J_l) \): there are extremal graphs with a more apparent structure, and they are often closer to the theoretical upper bound.
More specifically, some ``strictly smaller'' cases than \( \ram{5}{6} \) and \( \ram{5}{7} \) are still unsolved: \(30 \leq R(K_4,J_6)\leq 32 \) and \( 30 \leq R(J_5,K_5) \leq 33 \).

%--------------------------------------------------------------------------------------------------

\subsection*{Acknowledgements}

We would like to thank Gunnar Brinkmann and Stanis{\l}aw Radziszowski for useful suggestions.
Several of the computations for this work were carried out using the supercomputer infrastructure provided by the VSC (Flemish Supercomputer Center), funded by the Research Foundation Flanders (FWO) and the Flemish Government.

%--------------------------------------------------------------------------------------------------

\bibliographystyle{plain}
\bibliography{references}

%--------------------------------------------------------------------------------------------------

\newpage 

\section*{Appendix}

\begin{table}[h]
\centering
\scriptsize
\setlength{\tabcolsep}{4pt} %6 standard
\renewcommand{\arraystretch}{0.9}% Tighter
	\begin{tabular}{|r|cccccccccccccc|c|}
	\hline 
	$e$\textbackslash $n$ &3&4& 5 & 6 & 7 & 8 & 9 & 10 & 11 & 12 & 13 & 14 & 15 & 16 & total\\ 
	\hline
	0 & 1 & 1 & 1 &  &  &  &  &  &  &  &  &  &  &  & 3 \\ 
	1 & 1 & 1 & 1 &  &  &  &  &  &  &  &  &  &  &  & 3 \\ 
	2 & 1 & 2 & 2 & 2 &  &  &  &  &  &  &  &  &  &  & 7 \\ 
	3 & 1 & 3 & 4 & 5 & 1 &  &  &  &  &  &  &  &  &  & 14 \\ 
	4 &  & 2 & 6 & 9 & 6 & 1 &  &  &  &  &  &  &  &  & 24 \\ 
	5 &  &  & 5 & 14 & 16 & 2 &  &  &  &  &  &  &  &  & 37 \\ \hline
	6 &  &  & 3 & 17 & 34 & 15 & 1 &  &  &  &  &  &  &  & 70 \\ 
	7 &  &  &  & 12 & 49 & 49 & 4 &  &  &  &  &  &  &  & 114 \\ 
	8 &  &  &  & 6 & 55 & 122 & 25 & 1 &  &  &  &  &  &  & 209 \\ 
	9 &  &  &  & 2 & 45 & 210 & 101 & 5 &  &  &  &  &  &  & 363 \\ 
	10 &  &  &  &  & 22 & 260 & 355 & 23 & 1 &  &  &  &  &  & 661 \\ \hline
	11 &  &  &  &  & 6 & 223 & 853 & 104 & 3 &  &  &  &  &  & 1189 \\ 
	12 &  &  &  &  & 1 & 136 & 1399 & 529 & 12 & 1 &  &  &  &  & 2078 \\ 
	13 &  &  &  &  &  & 49 & 1537 & 2066 & 49 & 1 &  &  &  &  & 3702 \\ 
	14 &  &  &  &  &  & 12 & 1163 & 5567 & 230 & 4 &  &  &  &  & 6976 \\ 
	15 &  &  &  &  &  & 2 & 582 & 9713 & 1305 & 14 &  &  &  &  & 11616 \\ \hline
	16 &  &  &  &  &  & 1 & 187 & 11072 & 6876 & 45 &  &  &  &  & 18181 \\ 
	17 &  &  &  &  &  &  & 38 & 8261 & 24508 & 168 &  &  &  &  & 32975 \\ 
	18 &  &  &  &  &  &  & 9 & 4020 & 54803 & 912 &  &  &  &  & 59744 \\ 
	19 &  &  &  &  &  &  & 1 & 1238 & 76567 & 6341 &  &  &  &  & 84147 \\ 
	20 &  &  &  &  &  &  & 1 & 252 & 67697 & 36852 & 2 &  &  &  & 104804 \\ \hline
	21 &  &  &  &  &  &  &  & 41 & 37915 & 133255 & 26 &  &  &  & 171237 \\ 
	22 &  &  &  &  &  &  &  & 7 & 13360 & 288749 & 447 &  &  &  & 302563 \\ 
	23 &  &  &  &  &  &  &  & 2 & 2940 & 379164 & 5498 &  &  &  & 387604 \\ 
	24 &  &  &  &  &  &  &  & 1 & 420 & 306638 & 43510 &  &  &  & 350569 \\ 
	25 &  &  &  &  &  &  &  & 1 & 47 & 153238 & 196804 &  &  &  & 350090 \\ \hline
	26 &  &  &  &  &  &  &  &  & 4 & 47177 & 513057 & 40 &  &  & 560278 \\ 
	27 &  &  &  &  &  &  &  &  & 1 & 8832 & 786913 & 605 &  &  & 796351 \\ 
	28 &  &  &  &  &  &  &  &  &  & 1025 & 725109 & 6327 &  &  & 732461 \\ 
	29 &  &  &  &  &  &  &  &  &  & 78 & 405097 & 37163 &  &  & 442338 \\ 
	30 &  &  &  &  &  &  &  &  &  & 7 & 137389 & 128853 & 7 &  & 266256 \\\hline 
	31 &  &  &  &  &  &  &  &  &  &  & 28005 & 268857 & 24 &  & 296886 \\ 
	32 &  &  &  &  &  &  &  &  &  &  & 3420 & 343724 & 151 &  & 347295 \\ 
	33 &  &  &  &  &  &  &  &  &  &  & 250 & 269634 & 589 &  & 270473 \\ 
	34 &  &  &  &  &  &  &  &  &  &  & 16 & 129676 & 1645 &  & 131337 \\ 
	35 &  &  &  &  &  &  &  &  &  &  &  & 38220 & 3063 &  & 41283 \\ \hline
	36 &  &  &  &  &  &  &  &  &  &  &  & 6999 & 4105 &  & 11104 \\ 
	37 &  &  &  &  &  &  &  &  &  &  &  & 831 & 4030 &  & 4861 \\ 
	38 &  &  &  &  &  &  &  &  &  &  &  & 71 & 3156 &  & 3227 \\ 
	39 &  &  &  &  &  &  &  &  &  &  &  & 7 & 1979 &  & 1986 \\ 
	40 &  &  &  &  &  &  &  &  &  &  &  &  & 979 & 1 & 980 \\ \hline
	41 &  &  &  &  &  &  &  &  &  &  &  &  & 374 &  & 374 \\ 
	42 &  &  &  &  &  &  &  &  &  &  &  &  & 121 &  & 121 \\ 
	43 &  &  &  &  &  &  &  &  &  &  &  &  & 33 &  & 33 \\ 
	44 &  &  &  &  &  &  &  &  &  &  &  &  & 7 &  & 7 \\ 
	45 &  &  &  &  &  &  &  &  &  &  &  &  & 3 &  & 3 \\ \hline
	46 &  &  &  &  &  &  &  &  &  &  &  &  &  &  & 0 \\ 
	47 &  &  &  &  &  &  &  &  &  &  &  &  &  &  & 0 \\ 
	48 &  &  &  &  &  &  &  &  &  &  &  &  &  & 1 & 1 \\ 
	49 &  &  &  &  &  &  &  &  &  &  &  &  &  & 1 & 1 \\ 
	50 &  &  &  &  &  &  &  &  &  &  &  &  &  & 1 & 1 \\ 
	\hline
	total & 4& 9&22 & 67 & 235 & 1082 & 6256 & 42903 & 286738 & 1362501 & 2845543 & 1231007 & 20266 & 4 & 5796640 \\ 
	\hline
	\end{tabular}
	\caption{Counts for \( \RamSet (J_4,J_6) \) by number of vertices~$n$ and number of edges~$e$ in the first colour. }
	\label{table:countJ4J6}
\end{table}

\begin{table}[h]
	\centering
	\begin{tabular}{|r|ccc|}
		\hline
		$e$\textbackslash $n$ & 19 & 20 & 21 \\
		\hline 
		78 & 1 &  &   \\ 
		79 & 3 &  &   \\ 
		80 & 11 &  &   \\ 
		81 & 42 &  &  \\ 
		82 & 158 &  &  \\ 
		83 & 630 &  &  \\
		\hline 
		84 & 1926 &  & \\ 
		85 & 3440 &  & \\ 
		86 & 3440 &  &  \\ 
		87 & 1926 &  &  \\ 
		88 & 630 &  & \\ 
		89 & 158 &  & \\
		\hline
		90 & 42 & 3 & \\ 
		91 & 11 & 3 & \\ 
		92 & 3 & 5 &  \\ 
		93 & 1 & 11 & \\ 
		94 &  & 23 & \\ 
		95 &  & 35 &  \\
		\hline 
		96 &  & 23 &  \\ 
		97 &  & 11 &  \\ 
		98 &  & 5 & \\ 
		99 &  & 3 & \\ 
		100 &  & 3 & \\ 
		101 &  &  & \\
		\hline 
		102 &  &  & \\ 
		103 &  &  & \\ 
		104 &  &  &  \\ 
		105 &  &  & 1 \\ 
		\hline 
		total & 12422 & 125 & 1 \\ 
		\hline
	\end{tabular}
	\caption{Counts for \( \RamSet (J_5,J_5; n) \) for $19 \leq n \leq 21$ by number of vertices~$n$ and number of edges~$e$ in the first colour.}
	\label{table:countJ5J5}
\end{table}

\end{document}